\documentclass{amsart}
\usepackage{amsfonts}
\usepackage{amssymb,latexsym}
\usepackage{amsmath}
\usepackage{amssymb}
\usepackage[neverdecrease]{paralist}
\usepackage{pifont}

\newcommand{\vect}[1]{\ensuremath{\mathbf{#1}}} 
\newcommand{\card}[1]{\ensuremath{\lvert{#1}\rvert}} 
\newcommand{\minor}[3]{\ensuremath{{#1}_{{#2} \gets {#3}}}} 
\DeclareMathOperator{\ess}{ess} 
\DeclareMathOperator{\Ess}{Ess} 
\DeclareMathOperator{\gap}{gap} 
\DeclareMathOperator{\ord}{ord}
\DeclareMathOperator{\oddsupp}{oddsupp}

\newcommand{\Teta}{\ensuremath{\text{\large{\ding{100}}}}}

\theoremstyle{plain}
\newtheorem{theorem}{Theorem}[section]
\newtheorem{proposition}[theorem]{Proposition}

\newtheorem{corollary}[theorem]{Corollary}
\newtheorem{fact}[theorem]{Fact}

\theoremstyle{definition}
\newtheorem{example}[theorem]{Example}
\theoremstyle{remark}
\newtheorem{remark}[theorem]{Remark}

\begin{document}
\title{Additive decomposability of functions over abelian groups}
\author{Miguel Couceiro}
\address[M. Couceiro]{Mathematics Research Unit \\
University of Luxembourg \\
6, rue Richard Coudenhove-Kalergi \\
L-1359 Luxembourg \\
Luxembourg}
\email{miguel.couceiro@uni.lu}

\author{Erkko Lehtonen}
\address[E. Lehtonen]{Computer Science and Communications Research Unit \\
University of Luxembourg \\
6, rue Richard Coudenhove-Kalergi \\
L-1359 Luxembourg \\
Luxembourg}
\email{erkko.lehtonen@uni.lu}

\author{Tam\'as Waldhauser}
\address[T. Waldhauser]{Mathematics Research Unit \\
University of Luxembourg \\
6, rue Richard Coudenhove-Kalergi \\
L-1359 Luxembourg \\
Luxembourg
\and
Bolyai Institute \\
University of Szeged \\
Aradi v\'{e}rtan\'{u}k tere 1 \\
H-6720 Szeged \\
Hungary}
\email{twaldha@math.u-szeged.hu}

\begin{abstract}
Abelian groups are classified by the existence of certain additive decompositions of group-valued functions of several variables with arity gap $2$.
\end{abstract}
\maketitle

\section{Introduction}
\label{sect intro}

The arity gap of a function $f \colon A^n \to B$ is defined as the minimum decrease in the number of essential variables when essential variables of $f$ are identified. Up to the authors' knowledge, this notion first appeared in the 1963 paper by Salomaa~\cite{Salomaa}, where it was shown that the arity gap of every Boolean function is at most $2$. Willard~\cite{Willard} generalized this by showing that a function $f \colon A^n \to B$ defined on a finite set $A$ and depending on all of its variables has arity gap at most $2$ whenever $n > \max(\card{A}, 3)$; moreover, the arity gap of such a function equals $2$ if and only if $f$ is determined by $\oddsupp$ (see Section~\ref{sect prel} for definitions). Several papers on the topic have appeared ever since, e.g., \cite{CL2007,CLW2010b,GapA,Gappol,Shtrakov,SK2010}. A complete classification of functions according to their arity gap was presented in~\cite{CL2009}.

In a previous paper of the authors'~\cite{GapA}, unique additive decompositions of functions $f \colon A^n \to B$ were presented, assuming that $B$ has a group structure and the arity gap of $f$ is at least $3$. Similar decompositions were also proposed by Shtrakov and Koppitz~\cite{SK2010}. Further decompositions were also established in~\cite{GapA} for functions whose arity gap is $2$, but in this case the codomain $B$ was required to be a Boolean group. In the current paper,
we study similar additive decompositions of functions $f \colon A^n \to B$ into sums of functions with a smaller number of essential variables, assuming that $B$ is an abelian group. We show that such a decomposition exists for all functions $f \colon A^n \to B$ determined by $\oddsupp$ if and only if $A$ is finite and the exponent of $B$ is a power of $2$.

\section{Preliminaries}
\label{sect prel}

\subsection{Functions, essential variables, the arity gap}

Throughout this paper, let $A$ and $B$ be arbitrary sets with at least two elements. A \emph{partial function} (\emph{of several variables}) from $A$ to $B$ is a mapping $f \colon S \to B$, where $S \subseteq A^n$ for some integer $n \geq 1$, called the \emph{arity} of $f$. If $S = A^n$, then we speak of (\emph{total}) \emph{functions} (\emph{of several variables}). Functions of several variables from $A$ to $A$ are referred to as \emph{operations} on $A$.

For an integer $n \geq 1$, let $[n] := \{1, \dots, n\}$.
Let $f \colon S \to B$ ($S \subseteq A^n$) be an $n$-ary partial function and let $i \in [n]$.
We say that the $i$-th variable is \emph{essential} in $f$ (or $f$ \emph{depends} on $x_i$), if there exist tuples
\[
(a_1, \dots, a_{i-1}, a_i, a_{i+1}, \dots, a_n), (a_1, \dots, a_{i-1}, a'_i, a_{i+1}, \dots, a_n) \in S
\]
such that
\[
f(a_1, \dots, a_{i-1}, a_i, a_{i+1}, \dots, a_n) \neq f(a_1, \dots, a_{i-1}, a'_i, a_{i+1}, \dots, a_n).
\]
Variables that are not essential are called \emph{inessential.}
The cardinality of the set $\Ess f := \{i \in [n] : \text{$x_i$ is essential in $f$}\}$ is called the \emph{essential arity} of $f$ and is denoted by $\ess f$.

Let $f \colon A^n \to B$, $g \colon A^m \to B$. We say that $g$ is a \emph{simple minor} of $f$, if there is a map $\sigma \colon [n] \to [m]$ such that $g(x_1, \dots, x_m) = f(x_{\sigma(1)}, \dots, x_{\sigma(n)})$. We say that $f$ and $g$ are \emph{equivalent} if each one is a simple minor of the other.

For $i, j \in [n]$, $i \neq j$, define the \emph{identification minor} of $f \colon A^n \to B$ obtained by identifying the $i$-th and the $j$-th variable as the simple minor $\minor{f}{i}{j} \colon A^n \to B$ of $f$ corresponding to the map $\sigma \colon [n] \to [n]$, $i \mapsto j$, $\ell \mapsto \ell$ for $\ell \neq i$, i.e., $\minor{f}{i}{j}$ is given by the rule
\[
\minor{f}{i}{j}(x_1, \dots, x_n) := f(x_1, \dots, x_{i-1}, x_j, x_{i+1}, \dots, x_n).
\]

Observe that a function $g$ is a simple minor of $f$, if $g$ can be obtained from $f$ by permutation of variables, addition and deletion of inessential variables and identification of variables. Similarly, two functions are equivalent, if one can be obtained from the other by permutation of variables and addition of inessential variables.

The \emph{arity gap} of $f$ is defined as
\[
\gap f := \min_{\substack{i,j \in \Ess f \\ i \neq j}} (\ess f - \ess \minor{f}{i}{j}).
\]

Note that the definition of arity gap makes reference to essential variables only. Hence, in order to determine the arity gap of a function $f$, we may consider instead an equivalent function $f'$ that is obtained from $f$ by deleting its inessential variables. It is easy to see that in this case $\gap f = \gap f'$. Therefore, we may assume without loss of generality that every function the arity gap of which we may consider depends on all of its variables.

For general background and studies on the dependence of functions on their variables, see, e.g., \cite{Cimev1981,Cimev1986,Davies,DK,EKMM,Salomaa,Solovjev,Wernick,Yablonski}. For the simple minor relation and its variants, see, e.g., \cite{BCP,CP,EFHH,Hellerstein,Lehtonen2006,LS2011,Pippenger,Wang,Zverovich}. The notion of arity gap was considered in~\cite{CL2007,CL2009,CLW2010b,GapA,Gappol,Salomaa,Shtrakov,SK2010,Willard}, and a general classification of functions according to their arity gap was established in~\cite{CL2009}, given in terms of the notions of quasi-arity and determination by $\oddsupp$. The following explicit complete classification of Boolean functions was established in~\cite{CL2007}.

\begin{theorem}
\label{BooleanGap}
Let $f \colon \{0, 1\}^n \to \{0, 1\}$ be a Boolean function with at least two essential variables. Then\/ $\gap f = 2$ if and only if $f$ is equivalent to one of the following polynomial functions over $\mathrm{GF}(2)$:
\begin{enumerate}
\item $x_1 + x_2 + \dots + x_m + c$ for some $m \geq 2$,
\item $x_1 x_2 + x_1 + c$,
\item $x_1 x_2 + x_1 x_3 + x_2 x_3 + c$,
\item $x_1 x_2 + x_1 x_3 + x_2 x_3 + x_1 + x_2 + c$,
\end{enumerate}
where $c \in \{0,1\}$. Otherwise\/ $\gap f = 1$.
\end{theorem}

\subsection{Functions determined by $\oddsupp$}
\label{subsect oddsupp}

We will denote tuples by boldface letters and their components by the corresponding italic letters with subscripts, e.g., $\vect{x} = (x_1, \dots, x_n) \in A^n$. For $I \subseteq [n]$ and $\vect{x} \in A^n$, let $\vect{x}|_I \in A^I$ stand for the tuple that is obtained from $\vect{x}$ by deleting the $i$-th component of $\vect{x}$ for every $i \notin I$. More precisely, if $I = \{i_1, \dots, i_k\}$ and $i_1 < \dots < i_k$, then $\vect{x}|_I = (x_{i_1}, \dots, x_{i_k})$.

Berman and Kisielewicz~\cite{BK} introduced the following notion of a function's being determined by $\oddsupp$. Denote by $\mathcal{P}(A)$ the power set of $A$, and define the function $\oddsupp \colon \bigcup_{n \geq 1} A^n \to \mathcal{P}(A)$ by
\[
\oddsupp(a_1, \dots, a_n) :=
\{a \in A : \text{$\card{\{j \in [n] : a_j = a\}}$ is odd}\}.
\]
For $\varphi \colon \mathcal{P}( A) \to B$, let $\Teta_\varphi \colon \bigcup_{n \geq 1} A^n \to B$ be defined by $\Teta_\varphi(\vect{x}) = \varphi(\oddsupp(\vect{x}))$. A function $f \colon S \to B$ ($S \subseteq A^n$) is \emph{determined by $\oddsupp$} if $f(\vect{x})$ depends only on $\oddsupp(\vect{x})$, i.e., if there exists $\varphi \colon \mathcal{P}(A) \to B$ such that $\Teta_\varphi|_S = f$. When there is no risk of ambiguity, we will simply write $\Teta_\varphi$ instead of $\Teta_\varphi|_S$. Clearly, if $S = A^n$, then the restriction of $\varphi$ to
\[
\mathcal{P}'_n(A) = \bigl\{ S \in \mathcal{P}(A) : \card{S} \in \{n, n - 2, \ldots \} \bigr\}
\]
uniquely determines $f$ and vice versa. Thus, for finite sets $A$ and $B$, the number of functions $f \colon A^n \to B$ that are determined by $\oddsupp$ is $\card{B}^{\card{\mathcal{P}'_n(A)}}$. The following facts are straightforward to verify.

\begin{fact}
\label{fact Boolean}
The Boolean functions determined by $\oddsupp$ are exactly the affine functions (also known as linear functions in the theory of Boolean functions).
\end{fact}

\begin{fact}
\label{fact oddsupp}
A function $f \colon A^n \to B$ is determined by $\oddsupp$ if and only if $f$ is totally symmetric and $\minor{f}{2}{1}$ does not depend on $x_1$.
\end{fact}

\begin{fact}
\label{fact additive}
If $(B; +)$ is an abelian group, then $\Teta_{\varphi_1 + \varphi_2} = \Teta_{\varphi_1} + \Teta_{\varphi_2}$ holds for all maps $\varphi_1, \varphi_2 \colon \mathcal{P}(A) \to B$.
\end{fact}

It was shown by Willard \cite{Willard} that if the essential arity of a function $f \colon A^n \to B$ is sufficiently large, then $\gap f \leq 2$, and he also classified such functions according to their arity gap.

\begin{theorem}[Willard~\cite{Willard}]
\label{thm:gap}
Let $A$ be a finite set and $B$ be an arbitrary set, and assume that $f \colon A^n \to B$ depends on all of its variables and $n > \max(\card{A}, 3)$. If $f$ is determined by\/ $\oddsupp$ then\/ $\gap f = 2$. Otherwise\/ $\gap f = 1$.
\end{theorem}

If $B$ is a Boolean group (i.e., an abelian group of exponent $2$), then functions $f$ with $\gap f \geq 2$ can be characterized by the existence of certain additive decompositions. Here we present one of the main results of \cite{GapA} in the case $n > \max(\card{A}, 3)$. In this case, by Theorem~\ref{thm:gap}, $\gap f \geq 2$ if and only if $f$ is determined by $\oddsupp$.

\begin{theorem}[\cite{GapA}]
\label{thm oddsupp existence}
Let $(B; +)$ be a Boolean group, and let $f \colon A^n \to B$ be determined by $\oddsupp$. Then there exists a map $\varphi \colon \mathcal{P}'_n(A) \to B$ such that
\begin{equation}
f(\vect{x}) = \sum_{i = 1}^{\left\lfloor \frac{n}{2} \right\rfloor} \sum_{\substack{I \subseteq [n] \\ \card{I} = n - 2i}}\Teta_\varphi(\vect{x}|_I).
\label{eq fitilde}
\end{equation}
\end{theorem}

Equation~\eqref{eq fitilde} expresses the fact that every function $f \colon A^n \to B$ with large enough essential arity and $\gap f = 2$ is decomposable into a sum of essentially at most $(n - 2)$-ary functions. This fact is the starting point of the current paper. We will prove in Section~\ref{sect general} that such decompositions exist not only when $B$ is a Boolean group, but also whenever $B$ is a group whose exponent is a power of $2$. In fact, we will show that in this case there is a decomposition into functions with bounded essential arity, where the bound does not depend on $n$. We will also see that if the exponent of $B$ is not a power of $2$, then such a decomposition does not always exist, not even a decomposition into $(n - 1)$-ary functions. In Section~\ref{sect Boolean} we focus on Boolean groups $B$, and we provide a concrete decomposition of a very special symmetric form, which is also unique.

Any set $B$ can embedded into a Boolean group, e.g., into $\mathcal{P}(B)$ with the symmetric difference operation. Then we can regard any function $f \colon A^n \to B$ as a function from $A^n$ to $\mathcal{P}(B)$, and we can apply the results of Section~\ref{sect Boolean} to this function. We illustrate this for the case $A = B = \mathbb{Z}_3$ in Section~\ref{sect Z3}. Here we obtain decompositions involving a strange mixture of the field operations on $\mathbb{Z}_3$ and the symmetric difference operation, but we will see that they can be always computed within $B$, without the need of working in the extension $\mathcal{P}(B)$.

\subsection{Binomial coefficients}
\label{subsect binom}

We shall make use of the following combinatorial results.

\begin{theorem}[Shattuck, Waldhauser~\cite{SW}]
\label{thm binom SW}
For all nonnegative integers $m$, $t$ with $0 \leq t \leq \frac{m}{2} - 1$, the following identity holds:
\[
\sum_{i = t + 1}^{\left\lfloor \frac{m}{2} \right\rfloor} \binom{m}{2i} \binom{i - 1}{t}
= 2^{m - 2t - 1} \sum_{k = 0}^{\left\lfloor \frac{t}{2} \right\rfloor} \binom{m - 3 - t - 2k}{t - 2k} + (-1)^{t + 1}.
\]
\end{theorem}

\begin{theorem}
\label{thm binom trivi}
For all nonnegative integers $m$, $t$ with $0 \leq t \leq \frac{m - 1}{2}$ the following identity holds:
\[
\sum_{k = t + 1}^{\left\lfloor \frac{m + 1}{2} \right\rfloor} \binom{m}{2k - 1} \binom{2k - 1}{2t}
= \binom{m}{2t} 2^{m - 2t - 1}.
\]
\end{theorem}

\begin{proof}
Both sides of the identity count the number of pairs $(A, B)$, where $A \subseteq B \subseteq [m]$, $\card{A} = 2t$, and $\card{B}$ is odd.
\end{proof}

\section{The general case}
\label{sect general}

Throughout this section, unless mentioned otherwise, $A$ is a finite set with a distinguished element $0_A$ and $(B; +)$ is an arbitrary, possibly infinite abelian group with neutral element $0_B$. With no risk of ambiguity, we will omit the subscripts and will denote both $0_A$ and $0_B$ by $0$. Recall that the \emph{order} of $b \in B$, denoted by $\ord(b)$, is the smallest positive integer $n$ such that $nb = \underbrace{b + \dots + b}_{\text{$n$ times}} = 0$. If there is no such positive integer, then $\ord(b) = \infty$. If the orders of all elements of $B$ have a finite common upper bound, then the \emph{exponent} of $B$, denoted by $\exp(B)$, is the least common upper bound (equivalently, the least common multiple) of these orders. Otherwise let $\exp(B) = \infty$. Note that a Boolean group is a group of exponent $2$.

We say that a function $f \colon A^n \to B$ is \emph{$k$-decomposable} if it admits an additive decomposition $f = f_1 + \dots + f_s$, where the essential arity of each $f_i \colon A^n \to B$ is at most $k$. Moreover, we say that $f$ is \emph{decomposable} if it is $(n - 1)$-decomposable.

According to Fact~\ref{fact Boolean}, every Boolean function determined by $\oddsupp$ is $1$-decomposable, while the functions described in Theorem~\ref{thm oddsupp existence} are $(n - 2)$-de\-com\-pos\-able. Our goal in this section is to extend these results by characterizing those abelian groups $B$ which have the property that every function $f \colon A^n \to B$ determined by $\oddsupp$ is decomposable. As we will see, this is the case if and only if $\exp(B)$ is a power of $2$. Moreover, we will determine, for each such abelian group $B$, the smallest number $k$ such that every function $f \colon A^n \to B$ determined by $\oddsupp$ is $k$-decomposable.

The Taylor formula developed for finite functions by Gilezan \cite{Gilezan} provides a tool to test decomposability of functions. Although in \cite{Gilezan} the codomain $B$ was assumed to be a ring, only multiplication by $0$ and $1$ was used in the Taylor formula; hence it is valid for abelian groups as well. For self-containedness, we present here the formula with a proof (see Proposition~\ref{prop Taylor}).

For a given $\vect{x} \in A^n$ and $i \in [n]$, $a \in A$, let $\vect{x}_i^a$ denote the $n$-tuple that is obtained from $\vect{x}$ by replacing its $i$-th component by $a$. More generally, for $I \subseteq [n]$ and $\vect{a} \in A^n$, let $\vect{x}_I^{\vect{a}}$ denote the $n$-tuple that is obtained from $\vect{x}$ by replacing its $i$-th component by $a_i$ for every $i \in I$. (Observe that the components $a_i$ of $\vect{a}$ with $i \notin I$ are irrelevant in determining $\vect{x}_I^\vect{a}$.)

For any $a \in A$ and $i \in [n]$ we define the \emph{partial derivative} of $f \colon A^n \to B$ with respect to its $i$-th variable with parameter $a$ as the function $\Delta_i^a f \colon A^n \to B$ given by
\[
\Delta_i^a f(\vect{x}) = f(\vect{x}_i^a) - f(\vect{x}).
\]
Note that for each parameter $a \in A$ we have a different partial derivative of $f$ with respect to its $i$-th variable. We need the parameter $a$ because $A$ is just a set without any structure; hence we cannot define differences like $f(x + h) - f(x)$. It is easy to verify that the $i$-th variable of $f$ is inessential if and only if $\Delta_i^a f$ is identically
$0$ for some $a \in A$ (equivalently, for all $a \in A$).

Clearly, the partial derivatives are additive, i.e., $\Delta_i^a(f + g) = \Delta_i^a f + \Delta_i^a g$. Moreover, differentiations with respect to different variables commute with each other:
\begin{equation}
\Delta_i^a \Delta_j^b f(\vect{x}) = \Delta_j^b \Delta_i^a f(\vect{x}) = f(\vect{x}_{ij}^{ab}) - f(\vect{x}_i^a) - f(\vect{x}_j^b) + f(\vect{x})
\label{eq Delta2}
\end{equation}
for all $a, b \in A$, $i \neq j \in [n]$. (Here $\vect{x}_{ij}^{ab}$ is a shorthand notation for $(\vect{x}_i^a)_j^b = (\vect{x}_j^b)_i^a$.) This property allows us to define higher-order derivatives: for $I = \{i_1, \dots, i_k\} \subseteq [n]$ and $\vect{a} \in A^n$ let $\Delta_I^{\vect{a}} f = \Delta_{i_1}^{a_1} \cdots \Delta_{i_k}^{a_k} f$. (Again, the components $a_i$ ($i \notin I$) are irrelevant.) The following proposition generalizes formula~\eqref{eq Delta2} above.

\begin{proposition}
\label{prop DeltaI}
For any function $f \colon A^n \to B$, $I \subseteq [n]$ and $\vect{a} \in A^n$, we have
\[
\Delta_I^\vect{a} f(\vect{x}) = \sum_{J \subseteq I}(-1)^{\card{I \setminus J}} f(\vect{x}_J^\vect{a}).
\]
\end{proposition}

\begin{proof}
Easy induction on $\card{I}$. (For $\card{I} = 2$, the identity is just \eqref{eq Delta2}.)
\end{proof}

Now we are ready to state and prove the Taylor formula for functions $f \colon A^n \to B$, which is essentially the same as Theorem~2 and Theorem~3 in \cite{Gilezan}. (Let us note that in the following considerations any fixed $n$-tuple $\vect{a} \in A^n$ could be used instead of $\vect{0}$.)

\begin{proposition}
\label{prop Taylor}
Any function $f \colon A^n \to B$ can be expressed as a sum of some of its partial derivatives at\/ $\vect{0}$:
\begin{equation}
f(\vect{x}) = \sum_{I \subseteq [n]} \Delta_I^\vect{x} f(\vect{0}).
\label{eq Taylor}
\end{equation}
\end{proposition}

\begin{proof}
Using Proposition~\ref{prop DeltaI}, we can compute the right-hand side as follows:
\[
\sum_{I \subseteq [n]} \Delta_I^{\vect{x}} f(\vect{0}) = \sum_{I \subseteq [n]} \sum_{J \subseteq I}(-1)^{\card{I \setminus J}} f(\vect{0}_J^\vect{x}).
\]
Observe that $K := I \setminus J$ can be any subset of $[n] \setminus J$. Hence
\begin{align*}
\sum_{I \subseteq [n]} \sum_{J \subseteq I} (-1)^{\card{I \setminus J}} f(\vect{0}_J^{\vect{x}})
&= \sum_{J \subseteq [n]} \sum_{K \subseteq[n] \setminus J} (-1)^{\card{K}} f(\vect{0}_J^\vect{x}) \\
&= \sum_{J \subseteq [n]}
\Bigl( \sum_{K \subseteq [n] \setminus J} (-1)^{\card{K}} \Bigr) f(\vect{0}_J^\vect{x}).
\end{align*}
Since a nonempty finite set has the same number of subsets of odd cardinality as subsets of even cardinality, the coefficient $\sum_{K \subseteq [n] \setminus J} (-1)^{\card{K}}$ of $f(\vect{0}_J^\vect{x})$ above is $0$ unless $J = [n]$. Thus the sum reduces to $f(\vect{0}_{[n]}^\vect{x}) = f(\vect{x})$, and this proves the theorem.
\end{proof}

The following proposition provides a useful criterion of decomposability.

\begin{proposition}
\label{prop decomp.crit.}
A function $f \colon A^n \to B$ is $k$-decomposable if and only if $\Delta_I^\vect{a} f(\vect{0}) = 0$ for all\/ $\vect{a} \in A^n$ and $I \subseteq [n]$ with more than $k$ elements.
\end{proposition}

\begin{proof}
Sufficiency follows directly from Proposition~\ref{prop Taylor}: clearly, the essential arity of the function $\vect{x} \mapsto \Delta_I^\vect{x} f(\vect{0})$ is at most $\card{I}$. Therefore, if $\Delta_I^\vect{x} f(\vect{0})$ vanishes whenever $\card{I} > k$, then \eqref{eq Taylor} is a decomposition into a sum of essentially at most $k$-ary functions.

For necessity, let us suppose that $f = f_1 + \dots + f_s$, where $\ess f_i \leq k$ for $i \in [s]$. If $\card{I} > k$, then $I$ contains (the index of) at least one of the inessential variables of $f_i$, hence $\Delta_I^\vect{a} f_i$ is constant $0$ for every $\vect{a} \in A^n$ and $i \in [s]$. Since $\Delta_I^\vect{a} f = \Delta_I^\vect{a} f_1 + \dots + \Delta_I^\vect{a} f_s$, we can conclude that $\Delta_I^\vect{a} f$ is constant $0$ as well. In particular, we have $\Delta_I^\vect{a} f(\vect{0}) = 0$.
\end{proof}

The following two theorems constitute the main results of this section, and they show a strong dichotomy of abelian groups with respect to the decomposability of functions determined by $\oddsupp$.

\begin{theorem}
\label{thm decomp+}
If $A$ is a finite set and $B$ is an abelian group of exponent $2^e$, then every function $f \colon A^n \to B$ determined
by $\oddsupp$ is $(\card{A} + e - 2)$-decomposable.
\end{theorem}

\begin{proof}
Suppose that $f = \Teta_\varphi$ for some $\varphi \colon \mathcal{P}'_n(A) \to B$. By Proposition~\ref{prop decomp.crit.}, it suffices to verify that $\Delta_I^\vect{a} f(\vect{0}) = 0$ whenever $\card{I} \geq \card{A} + e - 1$. Let $\{a_i : i \in I\} =: \{b_1, \dots, b_t\}$ ($b_i \neq b_j$ whenever $i \neq j$), and let $B_j := \{i \in I : a_i = b_j\}$. Thus $\card{B_j}$ is the number of occurrences of $b_j$ in $\vect{a}|_I$; hence $\card{B_1} + \dots + \card{B_t} = \card{I}$ and $t \leq \card{A}$. Using Proposition~\ref{prop DeltaI}, we can expand $\Delta_I^\vect{a} f(\vect{0})$ as
\begin{equation}
\Delta_I^\vect{a} f(\vect{0}) = \sum_{J \subseteq I}(-1)^{\card{I \setminus J}} f(\vect{0}_J^\vect{a}) = \sum_{J \subseteq I}(-1)^{\card{I \setminus J}} \varphi(\oddsupp(\vect{0}_J^\vect{a})).
\label{eq:ph}
\end{equation}

Let us fix a set $S \subseteq A$ that appears as $\oddsupp(\vect{0}_J^\vect{a})$ in the above sum.

Assume first that $0 \in \{b_1, \dots, b_t\}$, say $b_t = 0$.
Then $\oddsupp(\vect{0}_J^\vect{a}) = S$ if and only if $\card{J \cap B_j}$ is odd whenever $b_j \in S$ and $\card{J \cap B_j}$ is even whenever $b_j \notin S$ for $j = 1, \dots, t - 1$ (note that $J \cap B_t$ is irrelevant in determining $\vect{0}_J^\vect{a}$). Since the number of subsets of $B_t$ of even cardinality equals the number of subsets of $B_t$ of odd cardinality, it holds that the number of sets $J$ satisfying $\oddsupp(\vect{0}_J^\vect{a}) = S$ that have an even cardinality equals the number of those that have an odd cardinality. Hence, the terms corresponding to such sets $J$ will cancel each other in~\eqref{eq:ph}.

Assume now that $0 \notin \{b_1, \dots, b_t\}$. Then clearly $t \leq \card{A} - 1$. Similarly, as in the previous case, we have that $\oddsupp(\vect{0}_J^\vect{a}) = S$ if and only if $\card{J \cap B_j}$ is odd whenever $b_j \in S$ and $\card{J \cap B_j}$ is even whenever $b_j \notin S$ for $j = 1, \dots, t$. Therefore, the number of sets $J \subseteq I$ satisfying $\oddsupp(\vect{0}_J^\vect{a}) = S$ is
\[
2^{\card{B_1} - 1} \cdots 2^{\card{B_t} - 1}
= 2^{\card{B_1} + \dots + \card{B_t} - t}
= 2^{\card{I} - t}.
\]
Moreover, the parity of $\card{J}$ is determined by $S$. Therefore, all occurrences of $\varphi(S)$ in~\eqref{eq:ph} have the same sign.

By the argument above, $\Delta_I^\vect{a} f(\vect{0})$ can be written as a sum of finitely many terms of the form $\pm 2^{\card{I} - t} \varphi(S)$, where $t \leq \card{A} - 1$. Since $\card{I} \geq \card{A} + e - 1$, the coefficient $2^{\card{I} - t}$ is a multiple of $2^e$; hence $\pm 2^{\card{I} - t} \varphi(S) = 0$ independently of the value of $\varphi(S)$. We conclude that $\Delta_I^\vect{a} f(\vect{0}) = 0$, as claimed.
\end{proof}

As the following example shows, Theorem~\ref{thm decomp+} cannot be improved and the number $\card{A} + e - 2$ cannot be decreased. More precisely, for every finite set $A$ with at least two elements, for every abelian group $B$ of exponent $2^e$, and for every $n > \card{A} + e - 3$, there exists a function $f \colon A^n \to B$ that is determined by $\oddsupp$ but is not $(\card{A} + e - 3)$-decomposable.

\begin{example}
Let $A = \{0, 1, \dots, \ell\}$, and let $B$ be an arbitrary abelian group of exponent $2^e$. Fix an element $b \in B$ of order $2^e$. Let $\varphi \colon \mathcal{P}(A) \to B$ be defined by
\[
\varphi(T) =
\begin{cases}
b, & \text{if $T \supseteq A \setminus \{0\}$,} \\
0, & \text{otherwise,}
\end{cases}
\]
let $n \geq \ell + e - 1$, and let $f \colon A^n \to B$ be given by $f(\vect{x}) = \Teta_\varphi(\vect{x})$.

To see that $f$ is not $(\card{A} + e - 3)$-decomposable, by Proposition~\ref{prop decomp.crit.}, it suffices to find $I \subseteq [n]$ and $\vect{a} \in A^n$ such that $\card{I} = \card{A} + e - 2 = \ell + e - 1$ and $\Delta_I^\vect{a} f(\vect{0}) \neq 0$. To this end, let
\[
\vect{a} := (1, 2, \dots, \ell - 1, \underbrace{\ell, \dots, \ell}_{e}, \underbrace{0, \dots, 0}_{n - \ell - e + 1}),
\]
and let $I := \{1, 2, \dots, \ell + e - 1\}$. Consider the expansion of $\Delta_I^\vect{a} f(\vect{0})$ as in~\eqref{eq:ph}. We can verify that for all $J \subseteq I$,
\[
f(\vect{0}_J^\vect{a}) =
\begin{cases}
b, & \text{if $J \supseteq \{1, \dots, \ell - 1\}$ and $\card{J \cap \{\ell, \dots, \ell + e - 1\}}$ is odd,} \\
0, & \text{otherwise.}
\end{cases}
\]
From this it follows that the number of sets $J \subseteq I$ satisfying $f(\vect{0}_J^\vect{a}) = b$ is $2^{e - 1}$. Therefore, we have
\[
\Delta_I^\vect{a} f(\vect{0}) = (-1)^{e - 1} 2^{e - 1} b \neq 0,
\]
where the inequality holds because the order of $b$ is $2^e$.
\end{example}

\begin{theorem}
\label{thm decomp-}
If $A$ is a finite set with at least two elements and $B$ is an abelian group whose exponent is not a power of\/ $2$, then for each $n$ there exists a  function $f \colon A^n \to B$ determined by\/ $\oddsupp$ that is not decomposable.
\end{theorem}

\begin{proof}
If the exponent of $B$ is not a power of $2$, then it has an element $b$ whose order is not a power of $2$ (possibly infinite). Let us consider first the special case $A = \{0, 1\}$. For any $\vect{x} \in A^n$ let $w(\vect{x})$ denote the \emph{Hamming weight} of $\vect{x}$, i.e., the number of $1$'s appearing in $\vect{x}$. Let $f_0 \colon A^n \to B$ be the function defined by
\[
f_0(\vect{x}) =
\begin{cases}
b, & \text{if $w(\vect{x})$ is even,} \\
0, & \text{if $w(\vect{x})$ is odd.}
\end{cases}
\]
Let us compute $\Delta_{[n]}^\vect{1} f_0(\vect{0})$ with the help of Proposition~\ref{prop DeltaI}:
\[
\Delta_{[n]}^\vect{1} f_0(\vect{0}) = \sum_{J \subseteq [n]}(-1)^{\card{[n] \setminus J}} f_0(\vect{0}_J^\vect{1}) = (-1)^n \sum_{J \subseteq [n]} (-1)^{\card{J}} f_0(\vect{0}_J^\mathbf{1}).
\]
Since $w(\vect{0}_J^\vect{1}) = \card{J}$, the above sum consists of $2^{n - 1}$ many $b$'s and $2^{n - 1}$ many $0$'s. Thus $\Delta_{[n]}^\vect{1} f_0(\vect{0}) = (-1)^n 2^{n - 1} b \neq 0$, as $\ord(b)$ does not divide $(-1)^n 2^{n - 1}$. Now Proposition~\ref{prop decomp.crit.} shows that $f_0$ is not $(n - 1)$-decomposable.

Considering the general case, let $0$ and $1$ be two distinguished elements of $A$, and let $f \colon A^n \to B$ be any function that is determined by $\oddsupp$ such that $f|_{\{0, 1\}^n} = f_0$. Then $f$ is not decomposable, since any decomposition of $f$ would give rise to a decomposition of $f|_{\{0, 1\}^n}$.
\end{proof}

\begin{corollary}
Let $A$ be a finite set with at least two elements, and $B$ be an abelian group. All functions $f \colon A^n \to B$ determined by $\oddsupp$ are decomposable if and only if the exponent of $B$ is a power of $2$.
\end{corollary}

As the following example shows, decomposability is not guaranteed when $A$ is infinite, no matter what the exponent of $B$ is.

\begin{example}
Let $A$ be an infinite set, let $B$ be an abelian group and let $0 \neq b \in B$. Fix $n \geq 2$, and let $S := \{s_1, \dots, s_n\} \subseteq A \setminus \{0\}$ with $\card{S} = n$. Define $f \colon A^n \to B$ by the rule
\[
f(\vect{x}) =
\begin{cases}
b, & \text{if $\{x_1, \dots, x_n\} = S$,} \\
0, & \text{otherwise.}
\end{cases}
\]
It is clear that $f$ is determined by $\oddsupp$. Computing $\Delta_{[n]}^\vect{a} f(\vect{0})$ for $\vect{a} := (s_1, \dots, s_n)$ as in~\eqref{eq:ph}, we obtain $\Delta_{[n]}^\vect{a} f(\vect{0}) = b \neq 0$. Hence $f$ is not decomposable by Proposition~\ref{prop decomp.crit.}.
\end{example}

\begin{remark}
Theorem~\ref{thm oddsupp existence} asserts that if $B$ is a Boolean group and $n > \card{A}$, then every function $f \colon A^n \to B$ determined by $\oddsupp$ is $(n - 2)$-decomposable. Theorem~\ref{thm decomp+} gives a stronger result as it provides a decomposition into a sum of functions whose essential arity has an upper bound that depends only on $A$ and $B$ (and not on $n$). Theorem~\ref{thm decomp-} implies that if $\exp(B)$ is not a power of $2$, then even the weakest kind of decomposability (namely, $(n - 1)$-decomposability) fails to hold for all functions $f \colon A^n \to B$ determined by $\oddsupp$.
\end{remark}

\section{The case of Boolean groups}
\label{sect Boolean}

In this section we assume that $A$ is a finite set with a distinguished element $0$ and $(B; +)$ is a Boolean group with neutral element $0$. Applying Theorem~\ref{thm decomp+} to this case (with $e = 1$), we see that every function $f \colon A^n \to B$ determined by $\oddsupp$ is $(\card{A} - 1)$-decomposable. Here we will provide a canonical, highly symmetric decomposition of such functions and show that it is unique.

If $n > \card{A}$, then Theorem~\ref{thm oddsupp existence} provides a decomposition of $f$ into a sum of functions of essential arity at most $n - 2$. Each summand $\Teta_\varphi(\vect{x}|_I)$ is a function determined by $\oddsupp$, and if $\card{I} > \card{A}$, then we can apply Theorem~\ref{thm oddsupp existence} to decompose $\Teta_\varphi(\vect{x}|_I)$ into a sum of functions of essential arity at most $\card{I} - 2$. Repeating this process as long as we have summands of essential arity greater than $\card{A}$, we end up with an $\card{A}$-decomposition of $f$. If the parities of $\card{A}$ and $n$ are different, then this is already an $(\card{A} - 1)$-decomposition. By counting how many times a given summand $\Teta_\varphi(\vect{x}|_I)$ appears, we arrive at decomposition \eqref{eq decomp odd} given below in Theorem~\ref{thm decomp odd}. If the parities of $\card{A}$ and $n$ are equal, then we have to further decompose the summands of essential arity $\card{A}$. We then get the more refined decomposition \eqref{eq decomp even} given below in Theorem~\ref{thm decomp even}. Note that in these theorems we assume that $B$ is finite. However, as we will see in Remark~\ref{rem infinite B}, the general case can be easily reduced to the case of finite groups.

\begin{theorem}
\label{thm decomp odd}
Let $f \colon A^n \to B$, where $B$ is a finite Boolean group, $A$ is a finite set, and $n - \card{A} = 2t + 1 > 0$. Then $f$ is determined by\/ $\oddsupp$ if and only if $f$ is of the form
\begin{equation}
f(\vect{x}) = \sum_{i = t + 1}^{\left\lfloor \frac{n}{2} \right\rfloor} \sum_{\substack{I \subseteq [n] \\ \card{I} = n - 2i}} \binom{i - 1}{t} \Teta_\varphi(\vect{x}|_I),
\label{eq decomp odd}
\end{equation}
for some map $\varphi \colon \mathcal{P}'_n(A) \to B$. Moreover, $\varphi$ is uniquely determined by $f$.
\end{theorem}

\begin{proof}
Let $g_\varphi \colon A^n \to B$ denote the function given by the right-hand side of \eqref{eq decomp odd}. Let us note that since $n > \card{A}$ and $n - \card{A}$ is odd, $\mathcal{P}'_n(A)$ contains all subsets of $A$ whose complement has an odd number of elements. Observe also that in~\eqref{eq decomp odd} $I$ ranges over subsets of $[n]$ of size $\card{A} - 1, \card{A} - 3, \ldots$; hence \eqref{eq decomp odd} provides an $(\card{A} - 1)$-decomposition of $f$. Clearly, for such sets $I$ we have $\oddsupp(\vect{x}|_I) \in \mathcal{P}'_n(A)$.

To prove the theorem, it suffices to show that the following three statements hold:%
\begin{enumerate}[\rm (i)]
\item \label{i decomp odd}
the number of functions $f \colon A^n \to B$ that are determined by $\oddsupp$ is the same as the number of maps $\varphi \colon \mathcal{P}'_n(A)\to B$;

\item \label{ii decomp odd}
$g_\varphi$ is determined by $\oddsupp$ for every $\varphi \colon \mathcal{P}'_n(A) \to B$;

\item \label{iii decomp odd}
if $\varphi_1 \neq \varphi_2$ then $g_{\varphi_1} \neq g_{\varphi_2}$.
\end{enumerate}
The existence and uniqueness of the decomposition then follows by a simple counting argument: the functions $f \colon A^n \to B$ determined by $\oddsupp$ are in a one-to-one correspondence with the functions $g_\varphi$. (Alternatively, the existence could be proved by repeated applications of Theorem~\ref{thm oddsupp existence}, as explained above.)

Statement~\eqref{i decomp odd} is clear: the number of functions $f \colon A^n \to B$ that are determined by $\oddsupp$ is $\card{B}^{\card{\mathcal{P}'_n(A)}}$, the same as the number of maps $\varphi \colon \mathcal{P}'_n(A)\to B$.

To see that~\eqref{ii decomp odd} holds, observe that each $g_\varphi$ is a totally symmetric function. Hence, by Fact~\ref{fact oddsupp}, it suffices to prove that $g_\varphi(x_1, x_1, x_3, \dots, x_n)$ does not depend on $x_1$. Let $\vect{x} = (x_1, x_1, x_3, \ldots, x_n)$ and let $I$ be a set appearing in the summation in \eqref{eq decomp odd} such that $1 \in I$ and $2 \notin I$. Then $I' : =I \bigtriangleup \{1, 2\} = (I \setminus \{1\}) \cup \{2\}$ ($\bigtriangleup$ denotes the symmetric difference) appears as well, since it has the same cardinality as $I$. As $\oddsupp(\vect{x}|_I) = \oddsupp(\vect{x}|_{I'})$, we have $\Teta_\varphi(\vect{x}|_I) = \Teta_\varphi(\vect{x}|_{I'})$, thus these two summands will cancel each other. The remaining sets $I$ either contain both $1$ and $2$ or neither of them. In the first case, $\oddsupp(\vect{x}|_I) = \oddsupp(\vect{x}|_{I \setminus \{1, 2\}})$, and hence $\Teta_\varphi(\vect{x}|_I)$ does not depend on $x_1$, whereas in the second case $x_1$ does not appear in $\Teta_\varphi(\vect{x}|_I)$ at all. Thus $g_\varphi(x_1, x_1, x_3, \dots, x_n)$ does not depend on $x_1$, which shows that \eqref{ii decomp odd} holds.

To prove \eqref{iii decomp odd}, suppose on the contrary that there exist maps $\varphi_1, \varphi_2 \colon \mathcal{P}'_n(A) \to B$ such that $\varphi_1 \neq \varphi_2$ but $g_{\varphi_1} = g_{\varphi_2}$. Then for $\varphi = \varphi_1 + \varphi_2$ we have $g_\varphi = g_{\varphi_1} + g_{\varphi_2} \equiv 0$ by Fact~\ref{fact additive}, that is,
\begin{equation}
\sum_{i = t + 1}^{\left\lfloor \frac{n}{2} \right\rfloor} \sum_{\substack{I \subseteq [n] \\ \card{I} = n - 2i}} \binom{i - 1}{t} \Teta_\varphi(\vect{x}|_I) = 0
\label{eq decomp unique odd}
\end{equation}
for all $\vect{x} \in A^n$. Moreover, since $\varphi_1 \neq \varphi_2$, there exists an $S \in \mathcal{P}'_n(A)$ with $\varphi(S) \neq 0$. Let us choose $S$ to be minimal with respect to this property, i.e., $\varphi(S) \neq 0$, but $\varphi$ vanishes on all proper subsets of $S$.

Suppose first that $S$ is nonempty, say $S = \{s_1, \dots, s_{n - 2r}\}$. Since $n - \card{A} = 2t + 1$, we have that $t \leq r - 1$. Let us examine the left-hand side of \eqref{eq decomp unique odd} for
\[
\vect{x} := (\underbrace{s_1, \ldots, s_1}_{2r+1}, s_2, \dots, s_{n - 2r}) \in A^n.
\]
Observe that $\oddsupp(\vect{x}|_I) \subseteq S$. If $\oddsupp(\vect{x}|_I) \subset S$, then $\Teta_\varphi(\vect{x}|_I) = 0$ by the minimality of $S$. If $\oddsupp(\vect{x}|_I) = S$, then $\Teta_\varphi(\vect{x}|_I) = \varphi(S) \neq 0$. The latter is the case if and only if $I$ is a proper superset of $\{2r + 2, \dots, n\}$ of cardinality $n - 2i$ for some $i$. The number of sets $I \subseteq [n]$ with $\card{I} = n - 2i$ and $I \supset \{2r + 2, \dots, n\}$ is $\binom{2r + 1}{2i}$. Hence the left-hand side of \eqref{eq decomp unique odd} equals
\[
\sum_{i = t + 1}^r \binom{2r + 1}{2i} \binom{i - 1}{t} \varphi(S).
\]
Since $r \geq t + 1$, the coefficient $\sum_{i = t + 1}^r \binom{2r + 1}{2i} \binom{i - 1}{t}$ of $\varphi(S)$ is odd according to Theorem~\ref{thm binom SW} (for $m = 2r + 1$). Therefore, taking into account that $B$ is a Boolean group, we can conclude that the left-hand side of \eqref{eq decomp unique odd} is $\varphi(S) \neq 0$, which is a contradiction.

Suppose then that $S$ is empty. Choose $\vect{x} := (s_1, \dots, s_1)$ for an arbitrary $s_1 \in A$. Since $S \in \mathcal{P}'_n(A)$, $n$ is even and hence each $I$ occurring in \eqref{eq decomp unique odd} is of even cardinality. Whenever $\card{I}$ is even, $\oddsupp(\vect{x}|_I) = \emptyset = S$ and $\Teta_\varphi(\vect{x}|_I) = \varphi(S)$. Therefore, the left-hand side of \eqref{eq decomp unique odd} becomes
\[
\sum_{i = t + 1}^{\left\lfloor \frac{n}{2} \right\rfloor} \binom{n}{2i} \binom{i - 1}{t} \varphi(S),
\]
which equals $\varphi(S)$ by Theorem~\ref{thm binom SW} (for $m = n$). This yields the desired contradiction, and the proof of \eqref{iii decomp odd} is now complete.
\end{proof}

\begin{theorem}
\label{thm decomp even}
Let $f \colon A^n \to B$, where $B$ is a finite Boolean group, $A$ is a finite set, and $n - \card{A} = 2t > 0$. Then $f$ is determined by\/ $\oddsupp$ if and only if $f$ is of the form
\begin{equation}
f(\vect{x}) = \sum_{i = t + 1}^{\left\lfloor \frac{n}{2} \right\rfloor} \sum_{\substack{I \subseteq [n] \\ \card{I} = n - 2i}}\binom{i - 1}{t} \Teta_\varphi(\vect{x}|_I) + \sum_{k = t + 1}^{\left\lfloor \frac{n + 1}{2} \right\rfloor} \sum_{\substack{K \subseteq [n] \\ \card{K} = n - 2k + 1}} \binom{2k - 1}{2t} \Teta_\varphi(\vect{x}|_K).
\label{eq decomp even}
\end{equation}
for some map $\varphi \colon \mathcal{P}(A) \to B$ satisfying $\varphi(S) = \varphi(S \bigtriangleup \{0\})$ for every $S \in \mathcal{P}(A)$. Moreover, $\varphi$ is uniquely determined by $f$.
\end{theorem}

\begin{proof}
Let us note first that since $n > \card{A}$ and $n - \card{A}$ is even, $\mathcal{P}'_n(A)$ contains all subsets of $A$ whose complement has an even number of elements. The number of maps $\varphi \colon \mathcal{P}(A) \to B$ satisfying $\varphi(S) = \varphi(S \bigtriangleup \{0\})$ for every $S \in \mathcal{P}(A)$ is $\card{B}^{\card{\mathcal{P}'_n(A)}}$, since $\varphi|_{\mathcal{P}'_n(A)}$ can be chosen arbitrarily, and this uniquely determines $\varphi|_{\mathcal{P}(A) \setminus \mathcal{P}'_n(A)}$. The number of functions $f \colon A^n \to B$ that are determined by $\oddsupp$ is $\card{B}^{\card{\mathcal{P}'_n(A)}}$ as well, and we can use the same counting argument as in the proof of Theorem~\ref{thm decomp odd}. The fact that the right-hand side of \eqref{eq decomp even} is determined by $\oddsupp$ can be proven in a similar way, and for the uniqueness it suffices to prove that if
\begin{equation}
\sum_{i = t + 1}^{\left\lfloor \frac{n}{2} \right\rfloor } \sum_{\substack{I \subseteq [n] \\ \card{I} = n - 2i}}\binom{i - 1}{t} \Teta_\varphi(\vect{x}|_I)
+ \sum_{k = t + 1}^{\left\lfloor \frac{n+1}{2} \right\rfloor} \sum_{\substack{K \subseteq [n] \\ \card{K} = n - 2k + 1}} \binom{2k - 1}{2t} \Teta_\varphi(\vect{x}|_K)
= 0
\label{eq decomp unique even}
\end{equation}
for all $\vect{x} \in A^n$, then $\varphi|_{\mathcal{P}'_n(A)}$ is identically $0$.

Suppose, for the sake of contradiction, that there exists an $S \in \mathcal{P}'_n(A)$ such that $\varphi(S) \neq 0$, and let $n - 2r$ be the cardinality of the smallest such $S$. If $r = t$, then $\varphi(A) = \varphi(A \setminus \{0\}) \neq 0$, and $\varphi$ is zero on all other subsets of $A$. Let $A = \{0, a_1, \dots, a_\ell\}$, where $\ell = n - 2t - 1$, and let $\vect{x} = (0, \dots, 0, a_1, \dots, a_\ell) \in A^n$, where the number of $0$'s is $2t + 1$. Then, for any set $I$ appearing in the first summation of \eqref{eq decomp unique even}, we have $A \setminus \{0\} \nsubseteq \oddsupp(\vect{x}|_I)$; hence $\Teta_\varphi(\vect{x}|_I) = 0$. Similarly, $\Teta_\varphi(\vect{x}|_K) = 0$ for all sets $K$ appearing in \eqref{eq decomp unique even}, except for $K = \{2t + 2, \dots, n\}$, where $\Teta_\varphi(\vect{x}|_K) = \varphi(A \setminus \{0\})$. Thus the left-hand side of \eqref{eq decomp unique even} equals $\varphi(A \setminus \{0\}) \neq 0$, contrary to our assumption.

Let us now consider the case $r > t$, and let us suppose first that there exists a set $S \in \mathcal{P}'_n(A)$ of cardinality $n - 2r$ such that $\varphi(S) \neq 0$ and $0 \in S$, say $S = \{s_1, \dots, s_{n - 2r}\}$ with $s_1 = 0$. Let $T$ be a subset of $S$. By the minimality of $\card{S}$, if $T \in \mathcal{P}'_n(A)$ then we have $\varphi(T) \neq 0$ if and only if $T = S$. Similarly, if $T \notin \mathcal{P}'_n(A)$ then we have $\varphi(T) \neq 0$ if and only if $T = S \setminus \{0\}$. (Indeed, if $T \neq S \setminus \{0\}$, then $T \bigtriangleup \{0\} \in \mathcal{P}'_n(A)$ is a proper subset of $S$. Hence $\varphi(T) = \varphi(T \bigtriangleup \{0\}) = 0$.)

Let us examine the left-hand side of~\eqref{eq decomp unique even} for
\[
\mathbf{x} := (\underbrace{s_1, \dots, s_1}_{2r + 1}, s_2, \dots, s_{n - 2r}) \in A^n.
\]
The same argument as in the proof of Theorem~\ref{thm decomp odd} shows that the first sum of~\eqref{eq decomp unique even} equals
\[
\sum_{i = t + 1}^r \binom{2r + 1}{2i} \binom{i - 1}{t} \varphi(S),
\]
which is $\varphi(S)$ by Theorem~\ref{thm binom SW}, since $r \geq t + 1$. If $K$ is a set of size $n - 2k + 1$ appearing in the second sum of~\eqref{eq decomp unique even}, then $\Teta_\varphi(\vect{x}|_K) = \varphi(S \setminus \{0\}) = \varphi(S)$ if $K \supseteq \{2r + 2, \dots, n\}$, and $\Teta_\varphi(\vect{x}|_K) = 0$ otherwise. The number of such sets $K$ is $\binom{2r + 1}{2k - 1}$, thus the second sum on the left-hand side of \eqref{eq decomp unique even} equals
\[
\sum_{k = t + 1}^{r + 1} \binom{2r + 1}{2k - 1} \binom{2k - 1}{2t} \varphi(S).
\]
By Theorem~\ref{thm binom trivi}, the coefficient of $\varphi(S)$ here is $\binom{2r + 1}{2t} 2^{2r - 2t}$, which is even since $r > t$. Thus the left-hand side of \eqref{eq decomp unique even} reduces to $\varphi(S)$, contradicting our assumption.

In the remaining case we have $r > t$ and for all $S \in \mathcal{P}'_n(A)$ of cardinality $n - 2r$ we have $0 \notin S$ whenever $\varphi(S) \neq 0$. Let $S = \{s_1, \dots, s_{n - 2r}\}$ be such a set, and let $T \subseteq S$. If $T \in \mathcal{P}'_n(A)$, then we have $\varphi(T) \neq 0$ if and only if $T = S$ by the minimality of $\card{S}$. Similarly, if $T \notin \mathcal{P}'_n(A)$, then we have $\varphi(T) = 0$. (Indeed, if $T \notin \mathcal{P}'_n(A)$ then $T \cup \{0\} = T \bigtriangleup \{0\} \in \mathcal{P}'_n(A)$ and $\card{T \bigtriangleup \{0\}} \leq \card{S}$. On the other hand, if $\varphi(T \bigtriangleup \{0\}) = \varphi(T) \neq 0$ then $\card{T \bigtriangleup \{0\}} \geq \card{S}$ by the minimality of $\card{S}$. Thus we have $\card{T \bigtriangleup \{0\}} = \card{S} = n - 2r$, hence $T \bigtriangleup \{0\}$ is a set in $\mathcal{P}'_n(A)$ with cardinality $n - 2r$ such that $\varphi(T \bigtriangleup \{0\}) \neq 0$ and $0 \in T \bigtriangleup \{0\}$, and then replacing $S$ by $T \bigtriangleup \{0\}$ we come back to the previous case.)

Let us choose $\mathbf{x} := (s_1, \dots, s_1, s_2, \dots, s_{n - 2r}) \in A^n$ as before, and examine the summands in \eqref{eq decomp unique even}. For each $K$ appearing in the second sum, $\oddsupp(\vect{x}|_K) \subseteq S$ and $\oddsupp(\vect{x}|_K) \notin \mathcal{P}'_n(A)$, thus $\Teta_\varphi(\vect{x}|_K) = 0$. For each $I$ appearing in the first sum, we have $\Teta_\varphi(\vect{x}|_I) = \varphi(S) \neq 0$ if $I$ is a proper superset of $\{2r + 2, \dots, n\}$; otherwise $\oddsupp(\vect{x}|_I) \subset S$, and so $\Teta_\varphi(\vect{x}|_I) = 0$. Therefore, using Theorem~\ref{thm binom SW} as before, we can conclude that the left-hand side of~\eqref{eq decomp unique even} equals $\varphi(S)$, and this contradiction finishes the proof of the theorem.
\end{proof}

\begin{remark}
\label{rem infinite B}
Theorems~\ref{thm decomp odd} and \ref{thm decomp even} still hold for infinite Boolean groups $B$. To see this, let $f \colon A^n \to B$ be a function that is determined by $\oddsupp$, where $A$ is a finite set and $B$ is a possibly infinite Boolean group, and let $R \subseteq B$ be the range of $f$. Since $R$ is finite, the subgroup $[R] \leq B$ generated by $R$ is also finite. (The free Boolean group on $r$ generators has cardinality $2^r$.) Applying Theorems~\ref{thm decomp odd} and \ref{thm decomp even} to $f \colon A^n \to [R]$, we obtain the desired decomposition of $f$. To show the uniqueness, suppose that $\varphi_1, \varphi_2 \colon \mathcal{P}(A) \to B$ both yield the function $f$. Then we can replace $B$ by its subgroup generated by the union of the ranges of $\varphi_1$ and $\varphi_2$, and apply the uniqueness parts of Theorems~\ref{thm decomp odd} and \ref{thm decomp even}.
\end{remark}


\section{Illustration: operations over the three-element set}
\label{sect Z3}

We saw in Theorem~\ref{BooleanGap} that a Boolean function of essential arity at least $4$ has arity gap $2$ if and only if it is a sum of essentially at most unary functions. Alternatively, this fact follows from the results of the previous section together with Willard's Theorem~\ref{thm:gap}. More generally, Theorems~\ref{thm decomp odd} and \ref{thm decomp even} can be applied to describe polynomial functions over finite fields of characteristic $2$ with arity gap $2$. In~\cite{Gappol} we provided a simpler and more explicit description of such polynomial functions. In this section we show how Theorems~\ref{thm decomp odd} and \ref{thm decomp even} can be used to describe functions $f \colon \mathbb{Z}_3^n \to \mathbb{Z}_3$ of arity at least $4$ with $\gap f = 2$. Since $\mathbb{Z}_3$ is not a Boolean group, we cannot apply these theorems directly. First we need to embed $\mathbb{Z}_3$ into a Boolean group. To this extent, let $A := \mathbb{Z}_3 = \{0, 1, 2\}$ with the usual field operations $+$ and $\cdot$, and $B := \mathcal{P}(A)$ with the symmetric difference operation $\oplus$. We use the notation $\oplus$ instead of the more common $\bigtriangleup$ in order to emphasize that this is a Boolean group operation on $B$ (which was denoted by $+$ before). The neutral element of $(A; +)$ is $0$, and the neutral element of $(B; \oplus)$ is the empty set $\emptyset$. We identify the elements of $A$ with the corresponding one-element sets, i.e., we simply write $a$ instead of $\{a\}$ for $a \in A$. In this way, $A$ becomes a subset (but, of course, not a subgroup) of $B$.

Let $f \colon A^n \to B$, where $n \geq 4$ is even. Then we have $n = 2t + 4$ in Theorem~\ref{thm decomp odd}, and the summation in \eqref{eq decomp odd} runs over the subsets of $[n]$ of size $2$ (for $i = t + 1$) and of size $0$ (for $i = t + 2$). The corresponding coefficients $\binom{i - 1}{t}$ are $\binom{t}{t} = 1$ and $\binom{t + 1}{t} = t + 1$, respectively. Thus $\binom{i - 1}{t} \Teta_\varphi(\vect{x}|_I) = \Teta_\varphi(\vect{x}|_I)$ whenever $\card{I} = 2$ or $I = \emptyset$ and $t$ is even (i.e., $n$ is divisible by $4$); on the other hand, if $I = \emptyset$ and $t$ is odd, then $\binom{i-1}{t} \Teta_\varphi(\vect{x}|_I) = 0$. Therefore, \eqref{eq decomp odd} takes one of the following two forms, depending on the residue of $n$ modulo $4$ (the summation indices $i$ and $j$ always run from $1$ to $n$, unless otherwise indicated):
\begin{align*}
f(\vect{x}) &= \bigoplus_{i < j} \varphi(\oddsupp(x_i, x_j)) \oplus \varphi(\emptyset) && \text{if $n \equiv 0 \pmod{4}$,} \\
f(\vect{x}) &= \bigoplus_{i < j} \varphi(\oddsupp(x_i, x_j)) && \text{if $n \equiv 2 \pmod{4}$.}
\end{align*}
(Note that $\varphi(\oddsupp(x_i, x_j)) = \varphi(\{x_i, x_j\})$ if $x_i \neq x_j$, and $\varphi(\oddsupp(x_i, x_j)) = \varphi(\emptyset)$ if $x_i = x_j$.)

If $n$ is odd, then we can apply Theorem~\ref{thm decomp even}. In this case we have $n = 2t + 3$, and in the first summation of \eqref{eq decomp even} $I$ is a one-element set ($i = t + 1$) and the corresponding coefficient is $\binom{i - 1}{t} = \binom{t}{t} = 1$. In the second summation, $K$ is either a two-element set ($k = t + 1$) or the empty set ($k = t + 2$). The corresponding coefficients $\binom{2k - 1}{2t}$ are $\binom{2t + 1}{2t} = 2t + 1$ and $\binom{2t + 3}{2t} = \frac{(2t + 3)(2t + 1)(t + 1)}{3} \equiv t + 1 \pmod{2}$. Thus, \eqref{eq decomp even} takes one of the following two forms:
\begin{align*}
f(\vect{x}) &= \bigoplus_{i < j} \varphi(\oddsupp(x_i, x_j)) \oplus \bigoplus_i \varphi(\{x_i\}) && \text{if $n \equiv 1 \pmod{4}$,} \\
f(\vect{x}) &= \bigoplus_{i < j} \varphi(\oddsupp(x_i, x_j)) \oplus \bigoplus_i \varphi(\{x_i\}) \oplus \varphi(\emptyset) && \text{if $n \equiv 3 \pmod{4}$.}
\end{align*}
(Note that $\varphi(\oddsupp(x_i)) = \varphi(\{x_i\})$.)

The above formulas are valid for any function $f \colon A^n \to B$, but we are interested only in functions whose range lies within $A$, i.e., whose values are one-element sets in $B$. In this case, we can give more concrete expressions for the above decompositions.

\begin{theorem}
\label{thm GF(3)}
Let $f \colon \mathbb{Z}_3^n \to \mathbb{Z}_3$ be a function of arity at least\/ $4$. Then\/ $\gap f = 2$ if and only if there exists a unary polynomial $p = ax^2 + bx + c \in \mathbb{Z}_3[x]$ and a constant $d \in \mathbb{Z}_3$, which are uniquely determined by $f$, such that
\begin{align*}
f(\vect{x}) &= \bigoplus_{i < j} \bigl( (x_i - x_j)^2 p(x_i + x_j) + d \bigr) \oplus d && \text{if $n \equiv 0 \pmod{4}$,} \\
f(\vect{x}) &= \bigoplus_{i < j} \bigl( (x_i - x_j)^2 p(x_i + x_j) + d \bigr) \oplus \bigoplus_i \bigl( p(x_i) + d \bigr) && \text{if $n \equiv 1 \pmod{4}$,} \\
f(\vect{x}) &= \bigoplus_{i < j} \bigl( (x_i - x_j)^2 p(x_i + x_j) + d \bigr) && \text{if $n \equiv 2 \pmod{4}$,} \\
f(\vect{x}) &= \bigoplus_{i < j} \bigl( (x_i - x_j)^2 p(x_i + x_j) + d \bigr) \oplus \bigoplus_i \bigl( p(x_i) + d \bigr) \oplus d && \text{if $n \equiv 3 \pmod{4}$.}
\end{align*}
Otherwise we have\/ $\gap f=1$.
\end{theorem}

\begin{proof}
Let $A := \mathbb{Z}_3$ and $B := \mathcal{P}(\mathbb{Z}_3)$ as explained above. We work out the details only for the case $n \equiv 3 \pmod{4}$, the other cases are similar. First let us consider the function
\[
f_1(\vect{x}) = \bigoplus_i \bigl( p(x_i) + d \bigr).
\]
It is clear that this function is totally symmetric, and $f_1(x_1, x_1, x_3, \dots, x_n)$ does not depend on $x_1$, since
\[
f_1(x_1, x_1, x_3, \dots, x_n)
= \bigl( p(x_1)+d \bigr) \oplus \bigl( p(x_1) + d \bigr) \oplus \bigoplus_{i = 3}^n \bigl( p(x_i) + d \bigr)
= \bigoplus_{i = 3}^n \bigl( p(x_i)+d \bigr).
\]
Therefore, $f_1$ is determined by $\oddsupp$ by Fact~\ref{fact oddsupp}. Hence $f_1(\vect{x}) = \varphi_{1}(\oddsupp(\vect{x}))$ for some map $\varphi_1 \colon \mathcal{P}'_n(A) \to B$. Observe that $\mathcal{P}'_n(A) = \{\{0\}, \{1\}, \{2\}, \{0, 1, 2\}\}$. Thus, in order to determine $\varphi_1$, it suffices to compute the following four values of $f_1$:
\begin{align*}
\varphi_1(\{0\})
&= f_1(0, \dots, 0)
= \bigoplus_{i = 1}^n \bigl( p(0) + d \bigr)
= p(0) + d
= c + d, \\
\varphi_1(\{1\})
&= f_1(1, \dots, 1)
= \bigoplus_{i = 1}^n \bigl( p(1) + d \bigr)
= p(1) + d
= a + b + c + d, \\
\varphi_1(\{2\})
&= f_1(2, \dots, 2)
= \bigoplus_{i = 1}^n \bigl( p(2) + d \bigr)
= p(2) + d
= a + 2b + c + d, \\
\varphi_1(\{0, 1, 2\})
&= f_1(0, \dots, 0, 1, 2)
= \bigoplus_{i = 1}^{n - 2} \bigl( p(0)+d \bigr) \oplus \bigl( p(1) + d \bigr) \oplus \bigl( p(2) + d \bigr) \\
&= \bigl( p(0) + d \bigr) \oplus \bigl( p(1) + d \bigr) \oplus \bigl( p(2) + d \bigr) \\
&= (c + d) \oplus (a + b + c + d) \oplus (a + 2b + c + d).
\end{align*}

We now analyze the function
\[
f_2(\vect{x}) = \bigoplus_{i < j} \bigl( (x_i - x_j)^2 p(x_i + x_j) + d \bigr)
\]
in a similar manner. Examining $f_2(x_1, x_1, x_3, \dots, x_n)$ we can see that the summands corresponding to $i = 1, j \geq 3$ cancel the summands corresponding to $i = 2, j \geq 3$, while the summand corresponding to $i = 1$, $j = 2$ is $(x_1 - x_1)^2 p(x_1 + x_1) + d = d$. Hence
\[
f_2(x_1, x_1, x_3, \dots, x_n)
= d \oplus \bigoplus_{3 \leq i < j} \bigl( (x_i - x_j)^2 p(x_i + x_j) + d \bigr),
\]
which clearly does not depend on $x_1$. Since $f_2$ is totally symmetric, we can conclude that $f_2$ is determined by $\oddsupp$. Therefore, there is a map $\varphi_2 \colon \mathcal{P}'_n(A) \to B$ such that $f_2(\vect{x}) = \varphi_2(\oddsupp(\vect{x}))$. For any $a \in A$ we have
\[
\varphi_2(\{a\})
= f_2(a, \dots, a)
= \bigoplus_{i < j} \bigl( (a - a)^2 p(a + a) + d \bigr)
= \binom{n}{2} d
= d,
\]
where the last equality holds, because $\binom{n}{2}$ is an odd number by the assumption that $n \equiv 3 \pmod{4}$.
To find $\varphi_2(\{0, 1, 2\})$, we can proceed as follows:
\begin{align*}
\varphi_2(\{0, 1, 2\})
&= f_2(0, \dots, 0, 1, 2) \\
&= \bigoplus_{i < j \leq n - 2} \bigl( (0 - 0)^2 p(0 + 0) + d \bigr)
\\
&\phantom{{} = {}} \qquad \oplus \bigoplus_{i = 1}^{n - 2} \bigl( (0 - 1)^2 p(0 + 1) + d \bigr) \oplus \bigoplus_{i = 1}^{n - 2} \bigl( (0 - 2)^2 p(0 + 2) + d \bigr)
\\
&\phantom{{} = {}} \qquad \oplus \bigl( (1 - 2)^2 p(1 + 2) + d \bigr)
\\
&= (a + b + c + d) \oplus (a + 2b + c + d) \oplus (c + d).
\end{align*}
(Here we made use of the fact that $\binom{n - 2}{2}$ is even and $n - 2$ is odd.)

The expression given for $f$ in the theorem is $f_1(\vect{x}) \oplus f_2(\vect{x}) \oplus d$, and from the above calculations it follows that this function is determined by $\oddsupp$, namely, $f_1(\vect{x}) \oplus f_2(\vect{x}) \oplus d = \varphi(\oddsupp(\vect{x}))$, where
\begin{align*}
\varphi(\{0\}) &= \varphi_1(\{0\}) \oplus \varphi_2(\{0\}) \oplus d = (c + d) \oplus d \oplus d = c + d, \\
\varphi(\{1\}) &= \varphi_1(\{1\}) \oplus \varphi_2(\{1\}) \oplus d = (a + b + c + d) \oplus d \oplus d = a + b + c + d, \\
\varphi(\{2\}) &= \varphi_1(\{2\}) \oplus \varphi_2(\{2\}) \oplus d = (a + 2b + c + d) \oplus d \oplus d = a + 2b + c + d, \\
\varphi(\{0, 1, 2\}) &= \varphi_1(\{0, 1, 2\}) \oplus \varphi_2(\{0, 1, 2\}) \oplus d \\
&= (c + d) \oplus (a + b + c + d) \oplus(a + 2b + c + d) \\
&\phantom{{} = (c + d)} \oplus (a + b + c + d) \oplus (a + 2b + c + d) \oplus (c + d) \oplus d = d.
\end{align*}
Observe that the range of $\varphi$ is a subset of $A$. Hence $f_1(\vect{x}) \oplus f_2(\vect{x}) \oplus d$ is a function from $A^n$ to $A$.

Let us consider the linear transformation%
\[
L \colon \mathbb{Z}_3^{4} \to \mathbb{Z}_3^4, \quad
(a, b, c, d) \mapsto (c + d, a + b + c + d, a + 2b + c + d, d).
\]
The determinant of $L$ is $1$; hence $L$ is a bijection. This means that the maps $\varphi \colon \mathcal{P}'_n(A) \to B$ that are of the above form are in a one-to-one correspondence with the $4$-tuples over $A$, i.e., there are $3^4 = 81$ such maps. The number of functions $f \colon A^n \to A$ that are determined by $\oddsupp$ is also $81$. Hence we can conclude by a simple counting argument that for any such $f$ there exists a unique tuple $(a, b, c, d)\in A^4$ such that $f(\vect{x}) = f_1(\vect{x}) \oplus f_2(\vect{x}) \oplus d$.
\end{proof}

Let us observe that when computing the value of a function of the form given in Theorem~\ref{thm GF(3)}, we do not have to ``leave'' $\mathbb{Z}_3$: using the fact that $\oplus$ is commutative and associative and it satisfies $u \oplus u \oplus v = v$ for any $u, v \in \mathbb{Z}_3$, we can always perform the calculations in such a way that we work only with singleton elements of $B$. It is not even necessary to know that $B$ is the power set of $\mathbb{Z}_3$, it could be any Boolean group that contains $\mathbb{Z}_3$ as a subset. To illustrate this point, let us compute $f(0, 0, 1, 2)$ for the function
\[
f(x_1, x_2, x_3, x_4)
= \bigoplus_{i < j} \bigl( (x_i - x_j)^2 p(x_i + x_j) + d \bigr) \oplus d
\]
that corresponds to the case $n = 4$ with $a = 1$, $b = c = d = 2$ in Theorem~\ref{thm GF(3)}:
\[
f(0, 0, 1, 2) = 2 \oplus 1 \oplus 0 \oplus 1 \oplus 0 \oplus 1 \oplus 2 =(0 \oplus 0) \oplus (1 \oplus 1) \oplus (2 \oplus 2) \oplus 1 = 1.
\]


\section*{Acknowledgments}

The third named author acknowledges that the present project is supported by the \hbox{T\'{A}MOP-4.2.1/B-09/1/KONV-2010-0005} program of National Development Agency of Hungary, by the Hungarian National Foundation for Scientific Research under grants no.\ K77409 and K83219, by the National Research Fund of Luxembourg, and cofunded under the Marie Curie Actions of the European Commission \hbox{(FP7-COFUND).}


\end{document}